\documentclass[12pt]{amsart}
\usepackage[all]{xy}
\usepackage{amsthm}
\usepackage{amsmath}
\usepackage{amssymb}
\usepackage{color}
\usepackage{float}
\usepackage{graphicx}
\usepackage[english]{babel}
\usepackage{verbatim}
\usepackage{enumerate}

\theoremstyle{plain}
\newtheorem{theorem}{Theorem}[section]
\newtheorem{definition}[theorem]{Definition}
\newtheorem{lemma}[theorem]{Lemma}

\newtheorem{solution*}{Solution}

\numberwithin{equation}{section}
\allowdisplaybreaks

\def\Der{\operatorname{Der}}
\def\span{\operatorname{span}}

\def\mg{\mathfrak{g}}
\def\mh{\mathfrak{h}}

\def\gl{\mathfrak{gl}}
\def\C{ \mathbb{C}}
\def\Z{ \mathbb{Z}}
\def\p{\partial}
\def\a{\alpha}
\def\b{\beta}

\def\span{\text{span}}

\def\Der{\text{Der}}
\def\mm{\mathfrak{m}}

\def\p{\partial}

\begin{document}
\title[Irreducible weight modules]{Irreducible Jet modules for
the vector field Lie algebra on $\mathbb{S}^1\times \mathbb{C}$}
\author{Mengnan Niu,  Genqiang Liu }

\begin{abstract}For a commutative algebra $A$ over $\mathbb{C}$,
 denote $\mathfrak{g}=\text{Der}(A)$. A module over  the smash product $A\# U(\mathfrak{g})$ is called a jet $\mathfrak{g}$-module, where $U(\mathfrak{g})$
 is the universal enveloping algebra of $\mathfrak{g}$.
In the present paper, we study jet modules in the case of $A=\mathbb{C}[t_1^{\pm 1},t_2]$.
 We show that  $A\#U(\mathfrak{g})\cong\mathcal{D}\otimes U(L)$,
  where $\mathcal{D}$ is the Weyl   algebra
$\mathbb{C}[t_1^{\pm 1},t_2, \frac{\partial}{\partial t_1},\frac{\partial}{\partial t_2}]$, and $L$ is a Lie subalgebra of  $A\# U(\mathfrak{g})$ called the jet Lie algebra corresponding to $\mathfrak{g}$. Using a Lie algebra isomorphism $\theta:L \rightarrow \mathfrak{m}_{1,0}\Delta$,  where $\mathfrak{m}_{1,0}\Delta$ is the subalgebra of
 vector fields vanishing at the point $(1,0)$,
we show that
any irreducible finite dimensional  $L$-module is isomorphic to an irreducible
$\gl_2$-module. As an application, we  give  tensor product realizations of  irreducible   jet modules over $\mathfrak{g}$ with uniformly bounded
weight spaces.
\end{abstract}
\date{}\maketitle

\vskip 5pt
\noindent
{\em 2020  Math. Subj. Class.:}
17B10, 17B20, 17B65, 17B66, 17B68

\vskip 10pt

\section{Introduction}

Among the theory of infinite dimensional Lie algebras,
 Lie algebras $\mathcal{V}_X$ of polynomial vector fields (i.e., the derivation Lie algebras of the affine coordinate algebras $A_X$)
 on irreducible affine algebraic varieties  $X$ is an important class of Lie
algebras.  This kind of Lie algebras have been studied
in the conformal field theory, see \cite{BP}. Unlike finite dimensional simple Lie algebras,
 the representation theory of vector fields Lie algebras at large is still not well developed.
 The centerless Virasoro algebra $W_1$ is the Lie algebra of polynomial vector fields on the
circle $\mathbb{S}^1$ whose  irreducible modules with
finite-dimensional weight spaces  were classified in \cite{M}.  Higher rank Witt
algebras $W_n$ are simple Lie algebras of polynomial vector fields on $n$-dimensional
torus.  There are quite a lot of
studies on representations  for $W_n$, see \cite{BMZ,BZ,E1,E2,EJ,GL,GLZ,GLLZ,GZ,LZ,MZ,TZ}.
 Billig and Futorny  classified all irreducible
Harish-Chandra $W_n$-module, see \cite{BF1}.  Weight modules for the Lie algebra of vector fields
on $\C^n$ were studied in \cite{CG,CoG,DSY,R1,R2,PS,XL}.
  Recently,  there is a systematic study
 on representations of the Lie algebra $\mathcal{V}_X$ for arbitrary smooth affine varieties $X$,
see \cite{BF2,BF3,BFN,BN,BIN}.

Let $A$ be a commutative associative algebra, $\mg$ the derivation algebra $\Der(A)$ of $A$.
 We say that a module $M$ is a jet module (also called $AV$-module in \cite{BFN,BIN})
  if it is a module both for the Lie algebra $\mg$ and for the commutative associative algebra $A$,
  and the action between $A$ and $\mg$ are compatible in the following rule:
$ X(fv)=f(Xv)+X(f)v, X\in\mg, f\in A, v\in M$. This definition of jet modules is
more general, without the condition that $M$ is free over $A$.
 The name of  jet modules (which are also free over the affine coordinate algebras) for vector field Lie algebras was firstly introduced in \cite{B}.
  Jet modules for  $\Der  \C[t_1^{\pm 1},\cdots, t_n^{\pm 1}]$
  were classified in \cite{E2}.  In \cite{B}, this classification was also given by the technique of polynomial modules. Since $A$ is a left  module algebra over
   the Hopf algebra $U(\mg)$, we have the smash product algebra $A\#U(\mg)$.
Then jet modules are exactly  modules over the associative algebra $A\# U(\mg)$.
 When $A=\C[x_1,\dots,x_n]$, the structure of the  algebra $A\# U(\mg)$
 was independently described by the classical Weyl algebra and the Lie algebra of
 vector fields vanishing at the origin in \cite{XL} and \cite{BIN}.
 We found that the ways of approaching jet modules
  for  $\Der  \C[t_1^{\pm 1},\cdots, t_n^{\pm 1}]$ and  $\Der  \C[t_1,\cdots, t_n]$ are different, see \cite{B,E2,BIN,XL}.
 Our motivation in the present paper is to find  a general method to
 handle jet modules for any vector field Lie algebra $\mathcal{V}_X$.
  As a tentative research, we study the structure of $A\#U(\mg)$
 for the mixed type algebra $\mathbb{C}[t_1^{\pm 1},t_2]$.
 We show that $A\#U(\mg)\cong \mathcal{D}\otimes U(L)$,
  where $\mathcal{D}$ is the differential operator  algebra
$\mathbb{C}[t_1,t_1^{-1},t_2, \frac{\partial}{\partial t_1},\frac{\partial}{\partial t_2}]$,
 and $L$ is the jet Lie algebra of $\mg$.  Using a Lie algebra isomorphism $\theta:L \rightarrow \mm_{1,0}\Delta$,  rather than  the technique of polynomial modules,
 we  show that any irreducible finite dimensional  $L$-module is actually an irreducible
$\gl_2$-module.
As a corollary, we give  a classification of  irreducible jet modules over $\mg$ with uniformly bounded weight spaces.
These known results reveal that the key point in researching  jet modules
for any vector field Lie algebra $\mathcal{V}_X$
 lies in  clarifying the
structure of the jet algebra associated with $\mathcal{V}_X$.

The paper is organized as follows. In Section 2,  we give some basic
facts about $\mg$ and our main result, see Theorem \ref{mainth}.
In Section 3, we study the structure of the jet Lie algebra $L$ for $\mg$.
We also classify finite dimensional irreducible modules over $L$, see Theorem \ref{L-module}. The  isomorphism $L \cong\mm_{1,0}\Delta$,   in Lemma \ref{lemma-hom} makes the structure of $L$ become very clear.
 In Section 4, we give a proof of Theorem \ref{mainth}.
   It should be noted that to separate the operator $\frac{\partial}{\partial t_1}$ in
   $\mathcal{D}=\mathbb{C}[t_1^{\pm 1},t_2, \frac{\partial}{\partial t_1},\frac{\partial}{\partial t_2}]$
   from $A\# U(\mg)$, we use $t_1^{-1}\cdot t_1\partial_1$ instead of
 the element $\partial_1$ in $\mg$.
 Finally, we give  tensor product realizations of  irreducible  jet modules over $\mg$ with uniformly bounded weight spaces, see Theorem \ref{th4.5}.

We denote by $\mathbb{Z}$, $\mathbb{Z}_+$, $\mathbb{N}$ and
$\mathbb{C}$ the sets of  all integers, nonnegative integers,
positive integers and complex numbers, respectively. For a Lie algebra
$L$ over $\C$, we use $U(L)$ to denote the universal enveloping algebra of $L$.

\section{Preliminaries  and main results}

We fix the vector space $\mathbb{C}^2$ of $2\times 1$ complex  matrices.
Denote the standard basis by $\{e_1,e_2\}$. For
$m=(m_1,m_2)\in \mathbb{Z}\times \mathbb{Z}_+$, denote $t^m=t_1^{m_1}t_2^{m_2}$.
Let $A=\mathbb{C}[t_1^{\pm 1},t_2], \mathfrak{g}=\text{Der}(A)$.
 Denote
$\p_i=\frac{\p}{\p t_i}$, $i=1,2$.
Then the algebra $\mg$ can be defined as follows:
\[\mg=\text{Span}\{t^{\a}\p_i|\a\in \Z \times \Z_+, i=1,2\}.\]
We can write the Lie bracket in $\mg$  as follows:
$$[t^{\a}\p_i, t^{\b}\p_j]=\b_it^{\a+\b-e_i}\p_j-\a_jt^{\a+\b-e_j}\p_i,
\forall\ \a, \b \in \Z \times \Z_+, 1\leq i,j\leq 2. $$
Note that the subspace $\mh=\text{Span}\{t_1\p_1,t_2\p_2\}$ is a Cartan subalgebra of $\mg$, i.e.,
a self-normalizing nilpotent  Lie subalgebra.

We denote the semidirect Lie algebra $\mg\ltimes A$ by $\widetilde{\mg}$.

\begin{definition}A $\widetilde{\mg}$-module $M$ is called a weight module if the action of $\mh$ on $M$
is diagonalizable, i.e.,
$M=\bigoplus_{\lambda \in \mh^{*}}M_{\lambda},$ where
$$M_{\lambda}=\{v \in M|(h-\lambda(h)Id)v=0,  \forall \ h \in \mh\}.$$
\end{definition}

\begin{definition}A module $M$ over $\widetilde{\mg}$ is a jet
$\mg$-module if the action  of $A$ is associative, i.e., $g(fv)=(gf)v$, for any $f,g\in A, v\in M$.
\end{definition}

Since  $A$ is a left  module algebra over the Hopf algebra $U(\mg)$, we have the smash product algebra $A\# U(\mg)$. A jet module is actually a module
over $A\#U(\mg)$.
We use $\cdot$ to denote the multiplication between $A$ and $U(\mg)$ in $A\sharp U(\mg)$.
To find homogenous elements in $A\# U(\mg)$ that  commute with
$t_1\cdot 1,t_2\cdot 1, 1\cdot t_1\partial_1, 1\cdot \partial_2$, we
define the following element:
 \begin{equation}\label{X}\aligned X_k(m)&=\sum_{i=0}^{m_2}(-1)^i\binom{m_2}{i}t_1^{-m_1}t_2^i\cdot t_1^{m_1+\delta_{k1}}t_2^{m_2-i}\partial_k-
 \delta_{m_2,0}1\cdot t_1^{\delta_{k1}}\partial_k,
 \endaligned\end{equation}where $m=(m_1,m_2)\in\Z\times \Z_+$, $k=1,2$.
 Note that $X_k(0,0)=0$ for $k=1,2$.

Let $L$ be the Lie subalgebra of $A\# U(\mg)$ spanned by $ X_k(m)$, for $k=1,2, m\in\Z\times \Z_{+}$,
 which is called the jet Lie algebra of $\mg$.  In the present paper,  we will show the following algebra  isomorphism.

\begin{theorem}\label{mainth}
The linear map $\phi:A\#U(\mg)\rightarrow \mathcal{D}\otimes U(L)$ defined by
$$\aligned\phi(1\cdot t^{m+\delta_{k1}e_1}\partial_k)&=
t_1^{m_1+\delta_{k1}}t_2^{m_2}\partial_k\otimes 1+\sum_{i=0}^{m_2}\binom{m_2}{i} t_1^{m_1}t_2^{i} \otimes X_k(m-ie_2),\\
\phi(t^m\cdot 1)&=t^m\otimes 1,\endaligned$$  is an associative algebra isomorphism, where $m\in\Z\times \Z_+$, $k=1,2$.
\end{theorem}

The isomorphism in Theorem \ref{mainth} tells us that there is  a close relationship
between vector field Lie algebras and Weyl algebras.

\section{The structure of $L$}
In this section, we study the structure of the Lie algebra $L$. We also classify finite dimensional irreducible modules over $L$.

In order to break up $A\#U(\mg)$  into
subalgebras: $\mathcal{D}$ and $ U(L)$,  which commute with each other, we need give the following lemma.

\begin{lemma}\label{D-L} For $m\in\Z\times \Z_+$, $k=1,2$ we have
\begin{enumerate}[$($a$)$]
\item  $[ X_k(m), t_1]=[ X_k(m), t_2]=0$.
\item $[ X_k(m), t_1\partial_1]=[ X_k(m), \partial_2]=0$.
\end{enumerate}
\end{lemma}

\begin{proof}(a). It is easy to see that $[X_k(m), t_i]=0$ if $k\neq i$.

For $k=1,2$, we have that
$$\aligned  \ [X_k(m), t_k]&=[\sum_{i=0}^{m_2}(-1)^i\binom{m_2}{i}t_1^{-m_1}t_2^i\cdot t_1^{m_1+\delta_{k1}}t_2^{m_2-i}\partial_k-
 \delta_{m_2,0}t_1^{\delta_{k1}}\partial_k,t_k]\\
&= \sum_{i=0}^{m_2}(-1)^i\binom{m_2}{i}t_1^{-m_1}t_2^i\cdot t_1^{m_1+\delta_{k1}}t_2^{m_2-i}-
 \delta_{m_2,0}t_1^{\delta_{k1}}\\
&= \sum_{i=0}^{m_2}(-1)^i\binom{m_2}{i}t_1^{\delta_{k1}}t_2^{m_2}-
 \delta_{m_2,0}t_1^{\delta_{k1}}=0.
\endaligned$$

(b). The equality  $[X_k(m), t_1\partial_1]=0$ follows from the fact that $X_k(m)$ is a homogenous element of degree zero with respect to the degree derivation $t_1\partial_1$.

From $(i+1)\binom{m_2}{i+1}=(m_2-i)\binom{m_2}{i}$,
we can check  that
$$\aligned  \ [X_k(m), \partial_2]
&=[\sum_{i=0}^{m_2}(-1)^i\binom{m_2}{i}t_1^{-m_1}t_2^i,\partial_2]\cdot t_1^{m_1+\delta_{k1}}t_2^{m_2-i}\partial_k\\
&\ \ +\sum_{i=0}^{m_2}(-1)^i\binom{m_2}{i}t_1^{-m_1}t_2^i\cdot [ t_1^{m_1+\delta_{k1}}t_2^{m_2-i}\partial_k,\partial_2]\\
&=\sum_{i=0}^{m_2}i(-1)^i\binom{m_2}{i}t_1^{-m_1}t_2^{i-1}\cdot t_1^{m_1+\delta_{k1}}t_2^{m_2-i}\partial_k\\
&\ \ -\sum_{i=0}^{m_2}(m_2-i)(-1)^i\binom{m_2}{i}t_1^{-m_1}t_2^i\cdot t_1^{m_1+\delta_{k1}}t_2^{m_2-i-1}\partial_k\\
&=0. \endaligned$$
We complete the proof of Lemma \ref{D-L}.
\end{proof}

By straightforward  calculations, we give the Lie bracket of $L$  in the following lemma.

\begin{lemma}\label{lemma-cr}For $m, s\in\Z\times \Z_+$, $k=1,2$, we have that
\begin{enumerate}[$($a$)$]
\item  $[X_1(m), X_1(s)]=m_1\delta_{s_2,0}X_1(m)-s_1\delta_{m_2,0}X_1(s)
    +(s_1-m_1)X_1(m+s)$.
\item $[X_2(m), X_2(s)]=-s_2\delta_{m_2,0}X_2(s-e_2)+m_2\delta_{s_2,0}X_2(m-e_2)+ (s_2-m_2)X_2(m+s-e_2).$
\item $[X_1(m), X_2(s)]= -s_1\delta_{m_2,0}X_2(s)+m_2\delta_{s_2,0}X_1(m-e_2)
+s_1 X_2(m+s)-m_2X_1(m+s-e_2)$.
\end{enumerate}
\end{lemma}

Let $\mathfrak{m}_{1,0}$ be the maximal ideal of  $A$ generated by
 $t_1-1,t_2$ and $\Delta=\span\{ \frac{\partial}{\partial t_1},\frac{\partial}{\partial t_2}\}$. Then $\mm_{1,0}\Delta$ is a Lie subalgebra of $\mg$.

Inspired by the isomorphism $\psi$ in \cite{XL2} for  Lie superalgebras, we have the following isomorphism for the Lie algebra $L$.

\begin{lemma}\label{lemma-hom} The linear map $\theta: L\rightarrow \mm_{1,0}\Delta$ defined by
$$\theta(X_1(m))=(t^{m}-\delta_{m_2,0}) t_1\frac{\partial}{\partial t_1},\ \
\theta(X_2(m))=(t^{m}-\delta_{m_2,0}) \frac{\partial}{\partial t_2},$$
 is a Lie algebra isomorphism, where $m=(m_1,m_2)\in\Z\times \Z_+,  k=1,2$.
\end{lemma}

\begin{proof}
For any $m,s\in \Z\times \Z_+$, we can check that
 $$\aligned \theta&([X_1(m), X_2(s)])\\
&= -s_1\delta_{m_2,0}(t^{s}-\delta_{s_2,0})\frac{\partial}{\partial t_2}+m_2\delta_{s_2,0}
(t^{m-e_2}-\delta_{m_2-1,0})t_1\frac{\partial}{\partial t_1}
\\
&\qquad +s_1 (t^{m+s}-\delta_{m_2+s_2,0})\frac{\partial}{\partial t_2}
 -m_2
(t^{m+s-e_2}-\delta_{m_2+s_2-1,0})t_1\frac{\partial}{\partial t_1}\\
&=s_1 (t^{m+s}-\delta_{m_2,0}t^s )\frac{\partial}{\partial t_2}-m_2(t^{s+m-e_2}-\delta_{s_2,0}t^{m-e_2})_1\frac{\partial}{\partial t_1}\\
&= [(t^{m}-\delta_{m_2,0}) t_1\frac{\partial}{\partial t_1},(t^{s}-\delta_{s_2,0}) \frac{\partial}{\partial t_2} ]\\
&= [\theta(X_1(m)), \theta(X_2(s))].
\endaligned$$

Similarly, we can verify that
$$\aligned\  & \theta([X_1(m), X_1(s)])
=[ \theta(X_1(m)),     \theta(X_1(s))],
        \endaligned $$
and
$$\aligned  & \theta([X_2(m), X_2(s)])
= [\theta(X_2(m)), \theta(X_2(s))].
\endaligned$$

Therefore the map $\theta$ is a homomorphism.
 Since $(t^{m}-\delta_{m_2,0}) t_1\frac{\partial}{\partial t_1}, (t^{m}-\delta_{m_2,0}) \frac{\partial}{\partial t_2}$ are linearly independent for different $m,s\in \Z\times \Z_+$,
$\theta$ is injective. Moreover we can see that $\mm_{1,0}\Delta$  is spanned by
 $(t^{m}-\delta_{m_2,0}) t_1\frac{\partial}{\partial t_1}, (t^{m}-\delta_{m_2,0}) \frac{\partial}{\partial t_2}$ with $m,s\in \Z\times \Z_+$. So $\theta$ is surjective, and hence it is an isomorphism.

\end{proof}

\begin{lemma}\label{iso-gl2} The linear map $\pi: \mm_{1,0}\Delta/\mm_{1,0}^2 \Delta \rightarrow \gl_2$ such that
 $$\aligned
&\pi( (t_1-1)\frac{\partial}{\partial t_1})= E_{11},
 \pi( (t_1-1)\frac{\partial}{\partial t_2}) = E_{12},\\
 &\pi(  t_2\frac{\partial}{\partial t_2})=E_{22},
  \pi(  t_2\frac{\partial}{\partial t_1})= E_{21},
 \endaligned$$
 is a Lie algebra isomorphism.
\end{lemma}

\begin{lemma}\label{lemma-ideal}
If $M$ is a finite dimensional irreducible $\mm_{1,0}\Delta$-module. Then $ \mm_{1,0}^2 \Delta M=0$. So $M $ is an irreducible $\gl_2$-module
via the isomorphism in Lemma \ref{iso-gl2}.
\end{lemma}

\begin{proof}  Let $A^+=\C[t_1,t_2], \mm^+_{1,0}=A^+\cap \mm_{1,0}$,
 and $d= (t_1-1)\frac{\partial}{\partial t_1}+t_2\frac{\partial}{\partial t_2}$.
 Then $\mm^+_{1,0}=\bigoplus_{i\in \Z_+} (\mm^+_{1,0})_i$ is a $\Z_+$-graded Lie algebra with respect to the adjoint action of $d$, where
 $(\mm^+_{1,0})_i=\{X\in \mm^+_{1,0} \mid  [d,X]=iX\}$. More precisely,
   $$(\mm^+_{1,0})_i=\span\{(t_1-1)^{m_1}t_2^{m_2}
\mid m_1,m_2\in \Z_+, m_1+m_2=i \}.$$

   Since $M$ is  finite dimensional, the action of  $d$  on $M$ has finite eigenvalues, hence there is some integer $l \geq 2$ such that $(\mm^+_{1,0})_{l+i} \Delta M=0$ for any $i\in\Z_+$. So   $(\mm^+_{1,0})^{l} \Delta M=0$.  Let $I$ be the ideal of
  $\mm_{1,0}\Delta$ generated by $ (\mm^+_{1,0})^{l} \Delta$ . Then $I M=0$.

   For any $k\in \Z, u_1\in (\mm^+_{1,0})_1, v_l\in (\mm^+_{1,0})^{l} $, we have

   $$\aligned \ \  [t_1^{k}u_1 d, v_l d]-[t_1^{k} d, u_1v_l d]
   &=t_1^{k}u_1 d(v_l) d-v_l(d(t_1^k)u_1+t_1^ku_1)d\\
   & \quad -  t_1^{k} (u_1v_l+u_1d(v_l)) d+u_1v_ld(t_1^k)d\\
   &=-2t_1^ku_1v_ld\in I.
   \endaligned$$
   So $\mm_{1,0}^{l+1} d M=0$.
   Then  $[v_ld, u_1\frac{\partial}{\partial t_i}]+u_1\frac{\partial v_l}{\partial t_i} d=v_lu_1\frac{\partial }{\partial t_i}\in I$,  for any $ v_l\in (\mm^+_{1,0})^{l} ,i=1,2$.
   Hence $\mm_{1,0}^{l+1}  \Delta M=0$.
   Consequently $M$ is an irreducible module over the quotient algebra $\mm_{1,0}\Delta  /\mm_{1,0}^{l+1} \Delta $.
   Since $\mm_{1,0}^+\Delta  /(\mm^+_{1,0})^{l+1}  \Delta \cong \mm_{1,0}\Delta  /\mm_{1,0}^{l+1} \Delta $, the  module $M$ is also an irreducible $\mm_{1,0}^+\Delta  /(\mm^+_{1,0})^{l+1} \Delta$-module. Since the adjoint action of $d$ on   $\mm_{1,0}^+\Delta $ is diagonalizable, so is  the action of $d$ on $M$.  From that the eigenvalues of $\text{ad} d$ on $\mm_{1,0}^2\Delta $ are all positive, $\mm_{1,0}^2\Delta M$ is a proper submodule of $M$. The irreducibility of $M$ forces that $\mm_{1,0}^2\Delta M=0$. Then we can complete the proof.

\end{proof}
By Lemma \ref{lemma-hom} and Lemma \ref{iso-gl2},  any irreducible $\gl_2$-module
$V$ can be lifted to an $L$-module denoted by $V_{\gl_2}^L$ in the following way:
$$\aligned X_k(i,0)v=iE_{1k}v, \qquad X_k(i,1)=E_{2k}v,  \qquad X_k(m_1,m_2)v=0,
\endaligned$$ where $v\in V, i, m_1\in\Z, k=1,2, m_2\in\Z_{\geq 2}$.

By Lemma \ref{lemma-ideal}, we obtain the classification of
finite dimensional irreducible $L$-modules.

\begin{theorem}\label{L-module}
 Let $M$ be a finite dimensional irreducible $L$-module.
Then   $M\cong V_{\gl_2}^L$  for some irreducible $\gl_2$-module.
\end{theorem}

\section{The map $\varphi$ is an isomorphism}
In this section, we will  show that $\varphi$ is an isomorphism.
Then using the classification of irreducible finite dimensional
 modules over $L$, we give the classification of
irreducible jet  $\mg$-modules with finite dimensional
weight spaces.

\begin{lemma}\label{th-phi}
The linear map $\phi:A\# U(\mg)\rightarrow \mathcal{D}\otimes U(L)$ defined in Theorem \ref{mainth}  is an associative algebra homomorphism.
\end{lemma}

\begin{proof} Clearly, the restricted map $\phi|_A: A\rightarrow\mathcal{D}$ is a homomorphism.
To show that $\phi$ is  a homomorphism, we also should check that
 $\phi$ preserves the defining relations of
  $\mg$.

For any $m,s\in\Z\times \Z_+$, we have that
$$\aligned
 &[\phi(1\cdot t^{m+e_1}\partial_1),\phi(1\cdot t^s\partial_2)]\\
&=[ t^{m+e_1}\partial_1\otimes 1,t^s\partial_2\otimes 1]+\sum_{i=0}^{m_2}\binom{m_2}{i} [t_1^{m_1}t_2^{i} ,t^s\partial_2]\otimes X_1(m-ie_2),\\
&\ \  +\sum_{j=0}^{s_2}\binom{s_2}{j} [t_1^{m_1+1}t_2^{m_2}\partial_1,t_1^{s_1}t_2^{j}  ]\otimes X_2(s-je_2)\\
&\ \ +  \sum_{i=0}^{m_2}\sum_{j=0}^{s_2}\binom{m_2}{i}\binom{s_2}{j}t_1^{m_1+s_1}t_2^{i+j} \otimes [X_1(m-ie_2), X_2(s-je_2)]\\
&=[ t^{m+e_1}\partial_1\otimes 1,t^s\partial_2\otimes 1]-\sum_{i=0}^{m_2}\binom{m_2}{i}i t_1^{m_1+s_1}t_2^{s_2+i-1} \otimes X_1(m-ie_2),\\
&\ \  +\sum_{j=0}^{s_2}\binom{s_2}{j} s_1 t_1^{m_1+s_1}t_2^{m_2+j}  \otimes X_2(s-je_2)\\
&\ \ +  \sum_{i=0}^{m_2}\sum_{j=0}^{s_2}\binom{m_2}{i}\binom{s_2}{j}
t_1^{m_1+s_1}t_2^{i+j} \otimes \\
&\ \  \Big(-s_1\delta_{m_2-i,0}X_2(s-je_2)
+(m_2-i)\delta_{s_2-j,0}X_1(m-(i+1)e_2)\Big)\\
&\ \ +\sum_{i=0}^{m_2}\sum_{j=0}^{s_2}\binom{m_2}{i}\binom{s_2}{j}\sum_{j=0}^{s_2}t_1^{m_1+s_1}t_2^{i+j} \otimes\\
&\  \   \Big(s_1 X_2(m+s-(i+j)e_2)
-(m_2-i)X_1(m+s-(i+j+1)e_2)\Big)\\
&=[ t^{m+e_1}\partial_1\otimes 1,t^s\partial_2\otimes 1]\\
&\ \ +\sum_{i=0}^{m_2}\sum_{j=0}^{s_2}\binom{m_2}{i}\binom{s_2}{j}\sum_{j=0}^{s_2}t_1^{m_1+s_1}t_2^{i+j} \otimes\\
&\ \   \Big(s_1 X_2(m+s-(i+j)e_2)
-(m_2-i)X_1(m+s-(i+j+1)e_2)\Big)\\
&=[ t^{m+e_1}\partial_1\otimes 1,t^s\partial_2\otimes 1]
 +s_1\sum_{i=0}^{s_2+m_2}\binom{s_2+m_2}{i} t_1^{m_1+s_1}t_2^{i} \otimes X_2(s+m-ie_2)\\
&\ \ + m_2\sum_{i=0}^{s_2+m_2-1}\binom{s_2+m_2-1}{i} t_1^{m_1+s_1}t_2^{i} \otimes X_1(s+m-e_2-ie_2)\\
&=\phi([1\cdot t^{m+e_1}\partial_1,\phi(1\cdot t^s\partial_2)]).
\endaligned$$

Similarly, we  can check that
 \begin{equation}\label{4.1}\phi([1\cdot t^{m+e_1}\partial_1,\phi(1\cdot t^s\partial_2)])=[\phi(1\cdot t^{m+e_1}\partial_1),\phi(1\cdot t^s\partial_2)],\end{equation}
\begin{equation}\label{4.3}\phi[1\cdot t_1^{m_1+1}t_2^{m_2}\partial_2,1\cdot t_1^{s_1+1}t_2^{s_2}\partial_2]=[\phi(1\cdot t_1^{m_1+1}t_2^{m_2}\partial_2),\phi(1\cdot t_1^{s_1+1}t_2^{s_2}\partial_2)].\end{equation}

Finally, we verify the commutation relation between $\mg$ and $A$.
Indeed,
$$\aligned \ &[\phi(1\cdot t^{m+\delta_{k1}e_1}\partial_k),\phi(t^s\cdot 1)] \\ &=[t_1^{m_1+\delta_{k1}}t_2^{m_2}\partial_k\otimes 1+\sum_{i=0}^{m_2}\binom{m_2}{i} t_1^{m_1}t_2^{i} \otimes X_1(m-ie_2), t^s\otimes 1]\\
&= s_k t^{m+s+\delta_{k1}e_1-e_k}\otimes 1=\phi([1\cdot t^{m+\delta_{k1}e_1}\partial_k,t^s\cdot 1]).
\endaligned$$

This completes the proof of Lemma \ref{th-phi}.
\end{proof}

\begin{lemma}\label{th-rho}
The linear map $\rho: \mathcal{D}\otimes U(L)\rightarrow A\# U(\mg)$ defined by
$$\aligned \rho(1\otimes X_k(m))&=\sum_{i=0}^{m_2}(-1)^i\binom{m_2}{i}t_1^{-m_1}t_2^i\cdot t_1^{m_1+\delta_{k1}}t_2^{m_2-i}\partial_k-
 \delta_{m_2,0}1\cdot t_1^{\delta_{k1}}\partial_k,\\
\rho(\partial_1\otimes 1)&=t_1^{-1}\cdot t_1\partial_1,\ \
\rho(\partial_2\otimes 1)=1\cdot \partial_2,\ \
\rho(t_k\otimes 1)=t_k\cdot 1,\endaligned$$  is an associative algebra homomorphism, where $m\in\Z\times \Z_+$, $k=1,2$. Moreover,
$\rho$ is the inverse of $\phi$.
\end{lemma}

\begin{proof} By the definition of $X_k(m)$ in (\ref{X}) and Lemma \ref{lemma-cr},
we can see that $\rho$ preserves the defining relations of $L$. To show that $\rho$ is a homomorphism, we also need to show
that $\rho$ preserves the defining relations of $\mathcal{D}$.
Indeed,
$$\aligned\ [\rho(\partial_k\otimes1),\rho(t_l\otimes 1)]
=[t_k^{-\delta_{k1}}\cdot t_k^{-\delta_{k1}}\partial_k,t_l\cdot 1]=\delta_{kl}=\rho([\partial_k\otimes1,t_l\otimes 1]).\endaligned$$

By Lemma \ref{D-L}, we see that $\rho$ preserves the commutativity relation between $L$ and $\mathcal{D}$.

By the definition, the composition $\varphi\rho$ is identity on $\mathcal{D}\otimes 1$. Let us check that it is also identity on $1\otimes U(L)$. We compute that
$$\aligned &\varphi\rho(1\otimes X_k(m))\\
&=\phi(\sum_{i=0}^{m_2}(-1)^i\binom{m_2}{i}t_1^{-m_1}t_2^i\cdot t_1^{m_1+\delta_{k1}}t_2^{m_2-i}\partial_k-
 \delta_{m_2,0}1\cdot t_1^{\delta_{k1}}\partial_k)\\
 &=\sum_{i=0}^{m_2}(-1)^i\binom{m_2}{i}(t_1^{-m_1}t_2^i\otimes 1)\Big(t_1^{m_1+\delta_{k1}}t_2^{m_2-i}\partial_k\otimes 1\\
 &\ \ +\sum_{j=0}^{m_2-i}\binom{m_2-i}{j} t_1^{m_1}t_2^{j} \otimes X_k(m-(i+j)e_2)\Big)-
\delta_{m_2,0}t_1^{\delta_{k1}}\partial_k\otimes 1\\
&=\sum_{i=0}^{m_2}(-1)^i\binom{m_2}{i}\sum_{j=0}^{m_2-i}\binom{m_2-i}{j} t_2^{i+j} \otimes X_k(m-(i+j)e_2)\\
&=\sum_{l=0}^{m_2}(\sum_{i=0}^{l}(-1)^i\binom{l}{i})\binom{m_2}{l} t_2^{l} \otimes X_k(m-le_2)\\
&=1\otimes X_k(m).
\endaligned$$

Clearly the composition $\rho\phi$ is identity  on $A$. We will check its value on $U(\mg)$. Explicitly,
$$\aligned &\rho\phi(1\cdot t^{m+\delta_{k1}e_1}\partial_k)\\
&=\rho(t_1^{m_1+\delta_{k1}}t_2^{m_2}\partial_k\otimes 1)+\sum_{i=0}^{m_2}\binom{m_2}{i} \rho(t_1^{m_1}t_2^{i} \otimes X_k(m-ie_2))\\
&= +\sum_{i=0}^{m_2}\binom{m_2}{i}\sum_{j=0}^{m_2-i}\binom{m_2-i}{j}(-1)^jt_2^{i+j}
\cdot t_1^{m_1+\delta_{k1}}t_2^{m_2-i-j}\partial_k\\
&=\sum_{l=0}^{m_2}(\sum_{i=0}^{l}(-1)^j\binom{l}{j})\binom{m_2}{l}
t_2^l\cdot t_1^{m_1+\delta_{k1}}t_2^{m_2-l}\partial_k\\
&= 1\cdot t^{m+\delta_{k1}e_1}\partial_k.
\endaligned$$

Therefore, $\rho$ is the inverse of $\phi$. The proof is complete.
\end{proof}

Combining Lemma \ref{th-phi} and Lemma \ref{th-rho}, we can establish Theorem \ref{mainth}. A $\mathcal{D}$-module $M$ is called a weight module if the actions of
$t_1\partial_1$ and $t_2\partial_2$ are diagonalizable.

\begin{lemma} \label{D-module}Any nonzero weight $\mathcal{D}$-module $M$ has an irreducible submodule.
\end{lemma}
\begin{proof} Choose a nonzero $v\in M$ such that $t_i\partial_iv=a_iv, i=1,2$ for some $a=(a_1,a_2)\in\C^2$. Let $I_a$ be the left ideal of $\mathcal{D}$ generate by $t_1\partial_1- a_1, t_2\partial_2- a_2$. We can see that $D/I_a\cong t^a\C[t_1^{\pm 1},t_2^{\pm 1}]$. If $a_2\in\Z$, then the $\mathcal{D}$-module $t^a\C[t_1^{\pm 1},t_2^{\pm 1}]$ has two irreducible sub-quotients: $t_1^{a_1}\C[t_1^{\pm 1},t_2]$, $t_1^{a_1}\C[t_1^{\pm 1},t_2^{\pm 1}]/t_1^{a_1}\C[t_1^{\pm 1},t_2]$. Otherwise, it is irreducible.
Therefore, as a quotient module of $D/I_a$, the submodule
$\mathcal{D}v$ of $M$ must has an irreducible submodule.
\end{proof}

 Let $P$ be a $\mathcal{D}$-module and $V$ be a $\gl_{2}$-module.
Then the tensor product $\mathcal{M}(P,V)=P\otimes_\C V$ becomes a $\widetilde{\mg}$-module (see \cite{Sh,LLZ})
with the action

\begin{equation}\label{Action1}
 t^{m+\delta_{k1}e_1}\partial_k \cdot(g\otimes v)=( t^{m+\delta_{k1}e_1}\partial_kg)\otimes v
+ m_1t^mg\otimes E_{1k}v+m_2t^{m-e_2}g\otimes E_{2k}v,
\end{equation}
\begin{equation}\label{Action2}
t^m \cdot(g\otimes v)=(t^m g)\otimes v,
\end{equation} where  $ t^m\in A, g\in P, v\in V,k=1,2$.
We will show any irreducible uniformly bounded jet modules for $\mg$ is
of the form $\mathcal{M}(P,V)$.

The following lemma is well known, see Lemma 2.7 in \cite{L}.

\begin{lemma}\label{tensor}Let $A, B$ be two unital associative algebras and $B$ has a countable basis. If $M$ is  an irreducible module   over $A\otimes B$ that  contains an irreducible $A=A\otimes \C$ submodule $W$, then $M\cong W\otimes V$ for  an irreducible $B$-module $V$.\end{lemma}

Finally, we will give a clear description of  all irreducible   jet $\mg$-modules with finite dimensional weight spaces.  An $L$-module $V$ is called a weight module if the actions of
$X_2(0,1)$ is diagonalizable.
Through the isomorphism $\phi:A\# U(\mg)\rightarrow \mathcal{D}\otimes U(L)$
in Theorem \ref{mainth}, an irreducible jet $\mg$-module $M$ can be viewed an irreducible module
 over the algebra $\mathcal{D}\otimes U(L)$, denoted by $M^{\phi}$.

\begin{theorem}\label{th4.5} Let $M$ be an irreducible   jet $\mg$-module with finite dimensional weight spaces. Then
$M^{\phi}  \cong P\otimes V$ for some irreducible weight $\mathcal{D}$-module $P$, some
 irreducible weight $L$-module $V$.  Moreover, if $M$ is uniformly bounded, then $M\cong \mathcal{M}(P,V)$, where $V$ is an irreducible  finite dimensional $\gl_2$-module.

\end{theorem}

\begin{proof}  Let $h_1=t_1\partial_1\otimes 1, h_2=t_2\partial_2\otimes 1+ 1\otimes X_2(0,1)$.  From
 $\phi(t_1\partial_1)=h_1, \phi(t_1\partial_2)=h_2$, we see that $M^{\phi}$ is a weight module with respect to the actions of $h_1$ and $h_2$. It can be seen that  the adjoint actions of $t_2\partial_2\otimes 1 $ and $1\otimes X_2(0,1)$ on  $\mathcal{D}\otimes U(L)$ is diagonalizable, and each weight space of $M^{\phi} $ is invariant under the actions
 of $t_2\partial_2\otimes 1 $ and $1\otimes X_2(0,1)$. So $M^{\phi}$ is a weight module over $\mathcal{D}$ and $1\otimes X_2(0,1)$ acts diagonally on  $M^{\phi}$.
  Then   from
 Lemma \ref{D-module} and Lemma \ref{tensor}, there is an
irreducible weight $\mathcal{D}$-module $P$, some
 irreducible weight $L$-module $V$ such that $M^{\phi}\cong P\otimes V$.
 If $M$ is uniformly bounded, then the $L$-module $V$ is  finite dimensional. By Theorem \ref{L-module},  $V$ is an irreducible $\gl_2$-module. Then we complete the proof.
\end{proof}

\begin{center}
\bf Acknowledgments
\end{center}

\noindent   G.L. is partially supported by NSF of China (Grants
11771122).


\vspace{4mm}

\noindent   \noindent M.N.: School of Mathematics and Statistics,
Henan University, Kaifeng 475004, China. Email: 18404904552@163.com

\vspace{0.2cm}

 \noindent G.L.: School of Mathematics and Statistics,
and  Institute of Contemporary Mathematics,
Henan University, Kaifeng 475004, China. Email: liugenqiang@henu.edu.cn

\end{document}